\documentclass[a4paper, 12pt]{amsart}
\usepackage[utf8]{inputenc} 
\usepackage[T1]{fontenc}       
\usepackage{lmodern}
  
\usepackage{graphicx}

\usepackage{amsfonts}
\usepackage{amsthm}
\usepackage{amsmath}
\usepackage{amssymb}
\usepackage{mathtools}
\usepackage{stackrel}

\usepackage{tikz-cd}
\usepackage[shortlabels]{enumitem}
\setenumerate{label=(\roman*)}
\usepackage[numbers]{natbib}
\usepackage{array}
\usepackage{xifthen}

\usepackage[top=3cm, bottom=3cm, left=3cm, right=3cm]{geometry}

\newcommand{\R}{\mathbb{R}}
\newcommand{\C}{\mathbb{C}}
\newcommand{\Q}{\mathbb{Q}}
\newcommand{\Z}{\mathbb{Z}}

\renewcommand{\H}{\mathbf{H}}
\newcommand{\calO}{\mathcal{O}}
\newcommand{\frakg}{\mathfrak{g}}

\newcommand{\inj}{\hookrightarrow}

\renewcommand{\phi}{\varphi}
\newcommand{\sep}{\; | \;}

\newcommand{\la}{\langle}
\newcommand{\ra}{\rangle}

\newcommand{\bs}{\backslash}

\newcommand{\dmo}[2][]{%
 \expandafter\DeclareMathOperator\csname#2\endcsname%
  {\ifthenelse{\isempty{#1}}%
    {#2}
    {#1}
  }%
}

\dmo{id}
\dmo{Ad}
\dmo{Isom}
\dmo{vol}
\dmo{SL}
\dmo{PSL}
\dmo{covol}
\dmo{PO}
\dmo{SO}
\dmo{GL}
\dmo{Mat}
\dmo[O]{OO}
\dmo{G}
\dmo{tr}

\theoremstyle{plain} \newtheorem*{theorem*}{Theorem}
\theoremstyle{plain} \newtheorem*{conjecture*}{Conjecture}
\theoremstyle{plain} \newtheorem*{lemma*}{Lemma}
\theoremstyle{plain} \newtheorem*{corollary*}{Corollary}
\theoremstyle{plain} \newtheorem*{proposition*}{Proposition}

\theoremstyle{plain} \newtheorem{thm}{Theorem}[section]
\theoremstyle{plain} 
\theoremstyle{plain} \newtheorem{lemma}[thm]{Lemma}
\theoremstyle{plain} \newtheorem{corollary}[thm]{Corollary}
\theoremstyle{plain} \newtheorem{proposition}[thm]{Proposition}
\theoremstyle{definition} 
\theoremstyle{definition} 

\theoremstyle{plain} \newtheorem{thmAbs}{Theorem}
\theoremstyle{plain} \newtheorem{theoremAbs}[thmAbs]{Theorem}
\theoremstyle{plain} 
\theoremstyle{plain} \newtheorem{corollaryAbs}[thmAbs]{Corollary}
\theoremstyle{plain} \newtheorem{propositionAbs}[thmAbs]{Proposition}
\theoremstyle{definition} 
\theoremstyle{definition} 

\theoremstyle{plain} 
\theoremstyle{plain} 
\theoremstyle{plain} \newtheorem*{claim*}{Claim}
\theoremstyle{definition} \newtheorem{remark}[thm]{Remark}
\theoremstyle{definition} 

\theoremstyle{plain} 

\begin{document}

\title 
{Nonarithmetic hyperbolic manifolds and trace rings}

\author{Olivier Mila}
\address{
  Universität Bern\\
  Mathematisches Institut (MAI)\\
  Alpeneggstrasse 22 \\
  3012 Bern \\
  Switzerland}

\email{olivier.mila@math.unibe.ch}

\begin{abstract}
  We give a sufficient condition on the hyperplanes used in the inbreeding construction of Belolipetsky-Thomson to obtain nonarithmetic manifolds.
  We construct explicitly infinitely many examples of such manifolds that are pairwise non-commensurable and estimate their volume.
\end{abstract}
\maketitle

Let $M$ be a finite-volume hyperbolic $n$-manifold, with $n \geq 2$.
If $M$ is complete, it can be written as a quotient $\Gamma \bs \H^n$ for $\Gamma$ a torsion-free lattice in the semi-simple Lie group 
$\PO(n,1) \cong \Isom(\H^n)$, 
the group of isometries of the hyperbolic $n$-space $\H^n$.
The lattice $\Gamma$ is uniform if and only if $M$ is compact.

A standard way to construct hyperbolic manifolds in higher dimensions is via arithmetic lattices.
In Lie groups with rank at least 2, this is actually the
only possible construction
(by Margulis' Arithmeticity Theorem);
yet it is known that there are nonarithmetic lattices in $\PO(n,1)$ for every $n \geq 2$.
Many examples in low dimensions were constructed by Vinberg using Coxeter groups (see \cite{Vinberg-non-arith-1} and the references therein), but the 
first construction in arbitrary dimension was given by Gromov and Piatetski-Shapiro \cite{GPS}.
Roughly, their idea consists in constructing two pieces of non-commensurable arithmetic manifolds with isometric boundaries and glueing them together to form a nonarithmetic manifold.
This construction has then been generalized by Raimbault \cite{Raimbault} and Gelander and Levit \cite{GL} to produce many different commensurability 
classes of nonarithmetic manifolds.

A similar construction was introduced by Belolipetsky and Thomson \cite{BT} to obtain manifolds with short systole.
They start with two hyperplanes chosen at distance $\delta > 0$ and find a torsion-free arithmetic lattice $\Gamma$ such that, in 
$M= \Gamma \bs \H^n$, the hyperplanes project down to two disjoint hypersurfaces.
Then they cut $M$ open along the hypersurfaces and glue it back to a copy of itself along its boundary;
as $\delta \to 0$, the systole of such a manifold then becomes arbitrarily small.
Manifolds obtained via this construction will be referred to as \emph{doubly-cut glueings} and the two corresponding 
hyperplanes as the \emph{cut hyperplanes} (see Section \ref{sec:nonarith}).

An interesting consequence is that infinitely many of such doubly-cut glueings are nonarithmetic and pairwise non-commensurable.
Moreover, these are the first examples in arbitrary dimension of nonarithmetic manifolds that are quasi-arithmetic 
(see \cite{Thomson}).
However if one is only interested in constructing nonarithmetic manifolds, their proof is somehow nonexplicit in the sense that it 
relies on the systole argument for proving both nonarithmeticity and pairwise non-commensurability.
Furthermore, it is hard to give an estimate on the volume of one particular nonarithmetic manifold.

In this paper we give a sufficient condition on the cut hyperplanes to obtain nonarithmetic doubly-cut glueings.
Recall that the group $\PO(n,1) \cong \Isom(\H^n)$ has a natural matrix representation in
$\OO_f(\R) \subset \GL_{n+1}(\R)$ for $f = -x_0^2 + x_1^2 + \cdots + x_n^2$ the standard Lorentzian
quadratic form (see Section \ref{sec:background}).

\begin{propositionAbs} \label{prop:main}
  Let $M$ be a doubly-cut glueing with cut hyperplanes $R_1$ and $R_2$.
  Let $\rho_1, \rho_2 \in \OO_f(\R)$ denote the reflections in $R_1,R_2$ respectively.
  If the trace of $g = \rho_1 \rho_2$ is not an algebraic integer, then $M$ is nonarithmetic.
\end{propositionAbs}

In order to study the commensurability classes of doubly-cut glueings, we use an invariant
called the \emph{adjoint trace ring}.
This invariant was introduced by Vinberg \cite{Vinberg} as the minimal ring of definition of
a lattice $\Gamma$ (see Section \ref{sec:trace-ring}).
We first show that we can realize every finitely generated subring of $\Q$ as the adjoint trace ring of a doubly-cut glueing.
\begin{theoremAbs} \label{th:Z[1/d]}
  Let $n \geq 4$. For every square-free integer $d >1$, there exists a nonarithmetic lattice
  $\Gamma_d$ in $\PO(n,1)$ with adjoint trace ring $\Z[1/d]$.
\end{theoremAbs}
It follows from the construction that these lattices are nonuniform (see Remark \ref{rem:nonuniform}).
Since the adjoint trace ring is an invariant of the commensurability class, an immediate corollary
is the following.
\begin{corollaryAbs}
  The lattices $\Gamma_d$ and $\Gamma_{d'}$ are non-commensurable whenever $d \neq d'$.
\end{corollaryAbs}

Thus we are able to construct many non-commensurable nonarithmetic doubly-cut glueings avoiding the
nonexplicit systole argument. This implies that it is easier to estimate the volume of such a
particular example.
A general proposition about volumes of doubly-cut glueings as well as an example of relatively small
volume can be found in Section \ref{sec:volumes}.

The proof mainly relies on an observation about manifolds which admit a mirror symmetry in
two of their embedded hypersurfaces.
In the first section we introduce the necessary background and prove Proposition \ref{prop:main}.
In Section \ref{sec:non-com-ex} we explain how to proceed to obtain non-commensurable manifolds and prove Theorem \ref{th:Z[1/d]}.
Section \ref{sec:volumes} is devoted to volumes computations and some generalizations of Theorem \ref{th:Z[1/d]} are
discussed in Section \ref{sec:generalizations}.

\emph{Acknowledgments:} The author expresses his gratitude towards his PhD advisor Vincent Emery for his guidance and useful comments, 
as well as to Filip Misev, Jean Raimbault and Scott Thomson for helpful discussions.

\section{Background and nonarithmeticity}
\subsection{Background}\label{sec:background}
Let $k \subset \R$ be a totally real number field with ring of integers $\calO_k \subset k$.
Let $\G$ be an absolutely simple adjoint algebraic $k$-group, and write $\G(\calO_k)$ in $\G(k) \cap \GL_N(\calO_k)$ for an 
arbitrary embedding $\G \subset \GL_N$ (this group is well defined up to commensurability).
For $n \geq 4$, we say that $\G$ is \emph{admissible} if $\G(\R) \cong \PO(n,1)$ and $\G^\sigma(\R)$ is compact for any non-trivial embedding 
$\sigma:k \inj \R$.
In that case, $\G(\calO_k)$ is a lattice in $\G(\R)$.

A lattice $\Gamma \subset \PO(n,1)$ is called \emph{arithmetic} if there exists an admissible algebraic $k$-group $\G$ and an isomorphism 
$\phi:\PO(n,1)\to \G(\R)$ such that 
\begin{equation}
  \label{eq:1}
  \phi(\Gamma) \text{ is commensurable to } \G(\calO_k).
\end{equation}
Since $\G$ is adjoint, we have $\phi(\Gamma) \subset \G(k)$ (see \cite[Prop. 1.2]{BP}).
The lattice $\Gamma$ is called \emph{quasi-arithmetic} if the condition \eqref{eq:1} is replaced by the weaker requirement that 
$\phi(\Gamma) \subset \G(k)$.

Let $n \geq 4$, and let $f$ be a quadratic form of rank $n+1$ defined over $k$.
Assume that $f$ is \emph{admissible}, i.e., that:
\begin{enumerate}
\item $f$ has signature $(n,1)$ when seen over $\R$; \label{it:quad1}
\item $f^\sigma$ is positive definite for any non-trivial embedding $\sigma:k \inj \R$.\label{it:quad2}
\end{enumerate}
Condition \ref{it:quad1} implies the existence of an $\R$-equivalence
$f \cong -x_0^2 + x_1^2 + \cdots + x_n^2$.
Therefore we may identify the hyperbolic space $\H^n$ with the ``$f$-hyperbolic space'' 
\[
\H_f = \{x = (x_0, \dots, x_n) \in \R^{n+1} \; \sep\; f(x) = -1 \quad \text{and} \quad x_0 > 0\}.
\]
Let $\OO_f$ denote the algebraic $k$-group of $f$-orthogonal matrices and $\PO_f = \OO_f / \{\pm \id\}$ the associated 
projective orthogonal group.
Conditions \ref{it:quad1} and \ref{it:quad2} ensure that $\PO_f$ is an admissible algebraic group; the corresponding arithmetic groups are called 
\emph{of the first type}.

Instead of directly working in $\PO_f(\R)$, it is more convenient for our purposes to use the Lie group 
\[
\OO_f'(\R) = \{ g =(g_{i,j})_{0 \leq i,j \leq n} \in \GL_{n+1}(\R) \sep f \circ g = f \quad \text{and} \quad g_{0,0} > 0\}.
\]
This group acts on $\H_f$ and may be identified with the group $\Isom(\H^n)$ of isometries of the hyperbolic $n$-space.
Moreover there is an obvious isomorphism $\OO_f'(\R) \cong \PO_f(\R)$ which allows us to see elements of $\PO_f(\R)$ as matrices.
Since it is a matrix group, we can define $\OO_f'(A)$ for any ring or field $A \subset \R$ as
$\OO_f'(\R) \cap \GL_{n+1}(A)$.
In particular, the group $\OO_f'(\calO_k)$ is unambiguously defined and is an arithmetic subgroup of the first type.

With these identifications, a hyperplane of $\H_f$ is simply $\H_f$ intersected with a linear subspace of $\R^{n+1}$ of dimension $n$.
We will say that such a hyperplane is \emph{$k$-rational} (or more briefly \emph{rational}) if the corresponding subspace in $\R^{n+1}$ 
is the $f$-orthogonal complement of a vector $v$ in $k^{n+1}$ (or equivalently in ${\calO_k}^{n+1}$).
Observe that the reflection in such a hyperplane is an isometry of $\H_f$ which lies in $\OO_f'(k)$.

\subsection{Nonarithmeticity} \label{sec:nonarith}
We briefly recall the construction of \cite{BT} in order to prove Proposition \ref{prop:main}.
Let $f/k$ be an admissible quadratic form over a totally real number field $k$.
Let $R_1, R_2$ be two disjoint rational hyperplanes in $\H_f$ and let $\Lambda \subset \OO_f'(\calO_k)$ be a finite-index torsion-free subgroup 
such that:
\begin{enumerate}
\item for $i = 1,2$ and for any $\lambda \in \Lambda$, $\lambda R_i$ is either disjoint from $R_i$ or coincides with it;
\item for any $\lambda  \in \Lambda$, $\lambda R_1$ is disjoint from $R_2$.
\end{enumerate}
In that case, for $i = 1,2$ the orbit $\Lambda R_i = \{\lambda  R_i \sep \lambda \in \Lambda\}$ forms a collection of disjoint hyperplanes, 
and $\Lambda R_1 \cap \Lambda R_2 = \emptyset$.
It follows from \cite[Lemma 3.1]{BT} that such a subgroup $\Lambda \subset \OO_f'(\calO_k)$ always exists.

Let $L = \Lambda \bs \H_f$.
By construction, the two hyperplanes $R_1$ and $R_2$ project down in $L$ to two disjoint \emph{hypersurfaces} (by which we mean 
totally geodesic codimension one embedded submanifolds); we denote them by $N_1$ and $N_2$ respectively.
Choose a connected component $C$ of $L \setminus (N_1 \cup N_2)$ and denote by $M$ the manifold obtained by glueing two copies of $C$
to each other by identifying their boundaries (that is, $M$ is the ``double'' of $C$).
We will call $M$ a \emph{doubly-cut glueing}, and $R_1$ and $R_2$ the \emph{cut hyperplanes}.

Write $\Lambda^+$ for the image of $\pi_1(C)$ in $\Lambda$ via the isomorphism $\pi_1(L) \cong \Lambda$.
Let $D, D'$ denote the two copies of $C$, seen as submanifolds of $M$.
Since $M$ is complete, we may write it as $M = \Gamma \bs \H_f$ with $\Gamma \subset \OO_f'(\R)$.
Such a $\Gamma$ can be chosen to contain $\Lambda^+$ in such a way that the following diagram commutes:
\[
\begin{tikzcd}
  \Lambda^+ \rar[hook] & \Lambda\\
  \pi_1(D) \rar[hook] \uar["\cong"] & \pi_1(M) \uar["\cong"].
\end{tikzcd}
\]

We will use the following lemma, which gives a generating set for $\Gamma$.
\begin{lemma} \label{lem:generating-set}
  Let $\rho_1, \rho_2 \in \OO_f'(k)$ denote the reflections in $R_1, R_2$ respectively.
  There exists $\lambda_1, \lambda_2 \in \Lambda$ such that 
  \[
  \Gamma = \la \;  \Lambda^+ , \;\rho_1 \Lambda^+ \rho_1 ,\;
  \rho_1 \rho_2, \; \rho_1 \lambda_1 \rho_1 \lambda_1^{-1},\; \rho_1 \lambda_2 \rho_2 \lambda_2^{-1} \; \ra.
  \]
  In particular, $\Gamma \subset  \la \;\OO_f'(\calO_k), \rho_1, \rho_2 \;\ra$. 
\end{lemma}
\begin{proof}
  Topologically, $M$ consists of $D$ and $D'$ glued together along their boundaries.
  Now depending on whether $N_1$ (resp. $N_2$) separates $L$ or not, 
  $\partial D$ consists of one or two copies of $N_1$ (resp. $N_2$).
  Write $\partial D = N_{1,1} \cup N_{1,2} \cup N_{2,1} \cup N_{2,2}$ with $N_{i,j} \cong N_i$ and possibly $N_{i,1} = N_{i,2}$.
  We can assume that the hyperplane $R_i$ is a lift of $N_{i,1}$.
  Moreover since the boundary components of $C$ corresponding to $N_{i,1}$ and $N_{i,2}$ are identified in $L$, 
  we can find a (possibly trivial) $\lambda_i \in \Lambda$ such that $\lambda_i R_i$ is a lift of $N_{i,2}$.
\vspace{-1em}
\begin{figure}[h!]\centering
\def\svgwidth{250pt}
\let\oldsmash\smash
\renewcommand{\smash}{\tiny \oldsmash}
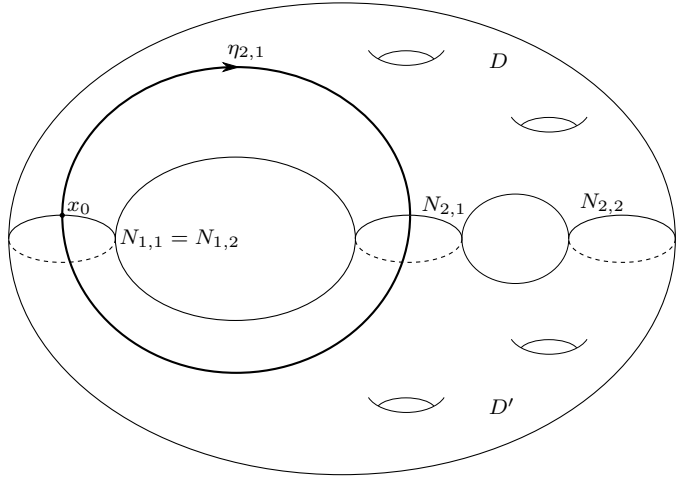
\caption{A doubly-cut glueing $M$.}
\end{figure}

  Observe that by construction the reflections in the $N_{i,j}$ are isometries of $M$, and that
  all these reflections have the
  same effect on $M$: they exchange $D$ and $D'$.
  From this it follows that for $i = 1,2$ and $\lambda \in \{1, \lambda_i\}$, the elements $\rho_1$ and $\lambda \rho_i \lambda^{-1}$ 
  induce the same isometry on $M$.
  Therefore 
  \[
  \rho_1 \lambda \rho_i \lambda^{-1} \in \Gamma.
  \]
  Now pick a base point $x_0$ in $N_{1,1}$.
  It is easy to see that $\rho_1 \rho_2 $ corresponds to the homotopy class
  $[\eta_{2,1}] \in \pi_1(M, x_0)$ of a loop $\eta_{2,1}$ at $x_0$ which crosses $N_{2,1}$ exactly once
  and is contained in $D$ before this crossing and in $D'$ after this crossing.

  Similarly, for $i=1,2$ the element $\rho_1 \lambda_i \rho_i \lambda_i^{-1}$ corresponds to the class of
  a loop $\eta_{i,2}$ crossing $N_{i,2}$ exactly once.
  Consequently, the group $\pi_1(M, x_0)$ is generated by 
  \[
  \pi_1(D) \cup \pi_1(D') \cup \{[\eta_{2,1}], [\eta_{1,2}] ,  [\eta_{2,2}]\}.
  \]
  Since $\pi_1(D')$ corresponds to $\rho_1 \Lambda^+ \rho_1$ in $\Gamma$, the lemma follows.
\end{proof}

\begin{remark}
  It is an easy consequence of the lemma that all doubly-cut glueings are quasi-arithmetic.
  This fact is also proven in \cite{Thomson}.
\end{remark}
  
We turn towards the proof of Proposition \ref{prop:main}.
\begin{proof}[Proof of Proposition \ref{prop:main}]
  We prove the contrapositive. Assume $M$ is arithmetic.
  Since $\Gamma$ and $\Lambda$ share the Zariski-dense subgroup $\Lambda^+$, it follows from 
  \cite[§1.6]{GPS} that $\Gamma \cap \OO_f'(\calO_k)$ is a finite index subgroup of $\Gamma$.
  Therefore the element $g = \rho_1 \rho_2 \in \Gamma$ must have a power $g^N$ in $\OO_f'(\calO_k)$.
  Let $\alpha$ be an eigenvalue of $g$.
  Then $\alpha^N$ is an eigenvalue of $g^N$, and is thus an algebraic integer (since it lies in an
  integral extension of $\calO_k$).
  It is easy to see that the same is true about $\alpha$.
  Since $\alpha$ is arbitrary, we get that $\tr(g) = \sum ($eigenvalues$)$ is also an algebraic
  integer.
\end{proof}

\begin{remark}\label{rem:trAd}
  The proof actually only uses the fact that $\rho_1\rho_2$ has an eigenvalue which is not an
  algebraic integer.
  Therefore the conclusion of the proposition still holds if the trace condition is replaced
  by this requirement.
\end{remark}

\section{Non-commensurability and examples over $\Q$} \label{sec:non-com-ex}
\subsection{The adjoint trace ring} \label{sec:trace-ring}
Proposition \ref{prop:main} gives a way to control the
nonarithmeticity of doubly-cut glueings, but does not say anything
about their commensurability.  To that end, we will use an invariant
introduced by Vinberg called the adjoint trace ring \cite{Vinberg}.

Let $\Gamma$ be a Zariski-dense subgroup of a semisimple algebraic
group $\G$.  The \emph{adjoint trace field} of $\Gamma$ is the field
\[
K(\Gamma) = \Q\left(\{\tr \Ad(\gamma) \sep \gamma \in
  \Gamma\}\right)
\]
where $\Ad$ denotes the adjoint representation.  Similarly, the
\emph{adjoint trace ring} $A(\Gamma)$ of $\Gamma$ is the integral
closure of the ring
$\Z\left[\{\tr \Ad(\gamma) \sep \gamma \in \Gamma\}\right]$.  If
$k = K(\Gamma)$ is a number field, we simply have (see
\cite{Davis}):
\[
A(\Gamma) = \calO_k\left[\{\tr \Ad(\gamma) \sep \gamma \in
  \Gamma\}\right].
\]

In his paper, Vinberg defines the minimal field (resp. ring) of definition of $\Ad \Gamma$, and 
shows that it coincides with $K(\Gamma)$ (resp. $A(\Gamma)$).
Yet we will not use this property in the rest of the article.  
The adjoint trace field is an invariant of the commensurability class of
$\Gamma$, and the same is true for the adjoint trace ring in case
the adjoint trace field is a number field \cite[Th. 3 and Cor. of
Th. 1]{Vinberg}.

Let us assume that $\G$ is defined over $\R$.
Suppose further that $\G(\R) \cong \PO(n,1)$ as Lie groups (with $n \geq 4$) and that
$\Gamma \subset \G(\R)$.  
For example, one can take $\G= \PO_f$ for $f$ a signature $(n,1)$ quadratic form.  
In that case, the algebraic adjoint representation and the one coming from the Lie group structure coincide on $\G(\R)$.  
Thus the adjoint trace field (resp. ring) of $\Gamma$ does not depend on $\G$, and it makes sense
to speak of the adjoint trace field $K(M)$ (resp. of the adjoint trace ring $A(M)$) of a hyperbolic manifold $M$.

\begin{remark} \label{rem:n=3-1} For $n=3$ the group $\Isom(\H^3) \cong \PSL_2(\C)$ has a structure
  of a complex Lie group.
  The adjoint trace field --- for the complex adjoint representation --- of a Zariski-dense
  $\Gamma \subset \PSL_2(\C)$ coincides with the invariant trace
  field in the sense of Maclachlan and Reid \cite[Def. 3.3.6]{MR}.
  If one uses an algebraic group $\G$ such that $\G(\R) \cong \Isom(\H^3)$, one gets a different adjoint trace field.  
  For $n\geq 4$, this confusion cannot happen, as $\PO(n,1)$ then does not possess a complex algebraic structure.
\end{remark}

By Borel's Density Theorem, any lattice $\Gamma \subset \G(\R)$ is Zariski-dense.
Furthermore, it follows from local rigidity that the adjoint
trace field of $\Gamma$ is a number field (even if $\Gamma$ is nonarithmetic, 
see \cite{Vinberg-Lie-Groups}).

The following easy lemma is useful to compute the adjoint trace ring.
\begin{lemma} \label{lem:trAd}
  Let $f$ be a quadratic form of rank $r$, and let $\OO_f(\C) \subset \GL_r(\C)$ denote the algebraic group of complex $f$-orthogonal matrices.
  Then for any $g \in \OO_f(\C)$,
  \[
  \tr \Ad(g) = \frac{(\tr g)^2 - \tr(g^2)}{2},
  \]  
  where $\Ad$ denotes the adjoint representation.
\end{lemma}
\begin{proof}
  Since all orthogonal groups are conjugate in $\GL_r(\C)$ and since the formula is invariant under conjugation, 
  it is enough to prove the lemma for the standard orthogonal group $\OO_r(\C)= \OO_f(\C)$ with $f = x_1^2 + \cdots + x_r^2$.
  
  Since $g$ is orthogonal, there exists an orthonormal basis $\{e_1, \dots, e_r\}$ of $\C^r$ which
  consists of eigenvectors of $g$.
  Let $\mu_i$ be the eigenvalue corresponding to $e_i$.
  The Lie algebra $\frakg$ is the subspace of $\Mat_{r}(\C)$ of skew-symmetric matrices.
  Therefore, the set
  \[
  \{ X_{i,j} = e_i e_j^t - e_j e_i^t  \sep 1 \leq i < j \leq r\}
  \]
  forms a basis of $\frakg$.
  
  For subgroups of $\GL_r(\C)$ the adjoint action coincides with
  conjugation, hence we have $\Ad(g)X_{i,j} = g X_{i,j} g^t = \mu_i \mu_j X_{i,j}$.
  Thus the $X_{i,j}$ are all eigenvectors of $\Ad(g)$, with corresponding eigenvalues $\mu_i \mu_j$.
  It follows that 
  \begin{equation}\label{eq:trAd-eigenvalues}
    \tr \Ad(g) = \sum_{1 \leq i < j \leq r} \mu_i \mu_j
  \end{equation}

  and one readily verifies that this equals $\frac{1}{2} ((\tr g)^2 - \tr(g^2))$.
\end{proof}

\begin{remark}
  Let $M$ be a doubly-cut glueing, with $\rho_1,\rho_2$ the two reflections
  corresponding to the cut hyperplanes, and assume that $\tr \Ad(\rho_1\rho_2) \notin \calO_k$.
  Then it follows already from \eqref{eq:trAd-eigenvalues} and Remark \ref{rem:trAd} that $M$ is
  nonarithmetic.
  Another way to see this is to observe that $\tr \Ad(\rho_1\rho_2)$ lies in the adjoint trace ring
  $A(M)$ of $M$ which is an invariant of the commensurability class.
  Since $A(M) \not \subset \calO_k$ it follows that $M$ cannot be commensurable to $\OO_f'(\calO_k)$,
  i.e., that $M$ is nonarithmetic.
  The advantage of this point of view is that it can be used to decide if two doubly-cut glueings
  are non-commensurable, as shown in the next section.
\end{remark}
\subsection{Examples over $\Q$}\label{sec:examples}
We will apply the results of the previous section to the quadratic form
$f = -x_0^2 + x_1^2 + \cdots + x_n^2$ defined over $\Q$.
In this section, norms, scalar products and orthogonal complements are to be understood with respect
to the quadratic form $f$.

Fix the hyperplane $R_1 = \{x_1 = 0 \} = v^\perp$ where $v = (0,1,0, \dots, 0)$.
Let $w \in \Z^{n+1}$ be such that 
\[
w_1^2 \geq \la w, w \ra > 0.
\]
Then the hyperplanes $R_1$ and $R_2= w^\perp$ do not intersect. 
Indeed we have $|w_1| = |\la v, w\ra | \geq \|v\| \|w\| = \sqrt{\la w, w \ra }$, and by \cite[Th 3.2.7]{Ratcliffe}
it follows that $R_1 \cap R_2 = \emptyset$.

Let $M_w$ be a doubly-cut glueing with cut hyperplanes $R_1$ and $R_2$.
Write $M_w = \Gamma_w \bs \H_f$ with $\Gamma_w \subset \OO_f'(\R)$.
Let $\rho_1, \rho_2$ denote the reflections in $R_1, R_2$ respectively.
\begin{lemma} \label{lem:Ad-formula}
  $\tr \Ad (\rho_1 \rho_2) \in  4(n-1)\frac{w_1^2}{\la w, w \ra} + \Z$.
\end{lemma}
\begin{proof}
  Let $g = \rho_1 \rho_2$.
  By Lemma \ref{lem:trAd} it suffices to compute $\frac{1}{2}((\tr g)^2- \tr(g^2))$.
  For $x \in \H_f$ we have $\rho_2(x) = x -2 \frac{\la x,w \ra}{\la w,w\ra} w$ and hence
  \[
    \rho_2 = I - \frac{2}{\la w, w \ra} w w^t J, \quad \text{where} \quad  J = (\la e_i, e_j\ra)_{i,j}.
  \]
  Now since 
  \[
  (\rho_1 w w^t J)^2 = \rho_1 w \underbrace{w^t J\rho_1 w}_{= \la w, \rho_1 w \ra} w^t J = \la w, \rho_1 w \ra \rho_1 w w^t J
  \]
  and since $g = \rho_1 \rho_2 = \rho_1 - \frac{2}{\la w,w\ra} \rho_1 ww^t J$, we have 
  \[
  g^2 = \rho_2 - \frac{2}{\la w,w \ra} \rho_1 w w^t J \rho_1 + 4 \frac{\la w, \rho_1 w \ra}{\la w, w\ra^2} \rho_1 w w^t J.
  \]
  Computing the trace then gives 
  \[
  \tr (g^2) = (n-1) -2 + 4 \frac{\la w, \rho_1 w \ra^2}{\la w, w\ra^2}.
  \]
  On the other hand $\tr g = (n-1) - 2 \frac{\la w, \rho_1 w \ra}{\la w, w \ra}$ whence 
  \[
  \frac{(\tr g)^2 - \tr (g^2)}{2} = \frac{(n-1)^2 - (n-1) -2}{2} -2(n-1)\frac{\la w, \rho_1 w \ra}{\la w, w \ra}.
  \]
  Since the numerator of the first fraction on the right hand side is even, and since $\la w, \rho_1 w \ra = \la w, w \ra -2 w_1^2$,
  this expression is indeed in $4(n-1)\frac{w_1^2}{\la w, w \ra} + \Z$.
\end{proof}
\begin{corollary} \label{cor:trace-ring}
  The adjoint trace ring $A(M_w)$ of the manifold $M_w = \Gamma_w \bs \H_f$ satisfies
  \[
  \Z\left[\frac{4(n-1) w_{1}^2}{\la w, w \ra}\right] \;\subset\; A(M_w)\; \subset\; \Z\left[\frac{2}{\la w, w\ra}\right].
  \]
  Furthermore, the inclusions are equalities if and only if $\Gamma_w \subset \OO_f'(A(M_w))$.
\end{corollary}
\begin{proof}
  The left inclusion is clear from the lemma, and the right inclusion follows from Lemma
  \ref{lem:generating-set} since $\OO_f'(\Z) \cup \{ \rho_1,\rho_2 \} \subset
  \OO_f'(\Z[2/ \la w, w \ra])$.
\end{proof}

We now dive into the proof Theorem \ref{th:Z[1/d]}.
\begin{proof}[Proof of Theorem \ref{th:Z[1/d]}]
  Write $d = p_1 \cdots p_r$, where the $p_i$ are distinct prime numbers.
  Assume the $p_i$'s are ordered such that $4(n-1) = a \cdot p_1^{e_1} \cdots p_s^{e_s}$ for some
  $s \leq r$, with $(a,d) = 1$.
  Let $b = p_1^{e_1 + 1} \cdots p_s^{e_s + 1} \cdot p_{s+1} \cdots p_r$.
  Using Bézout's Theorem and elementary properties of subrings of $\Q$, it is easy to see that
  \[
    \Z[{4(n-1)}/{b}] \quad =\quad  \Z[1/d] \quad =\quad  \Z[{2}/{b}].
  \]
  
  Now pick $w_1$ such that $w_1^2 > b$ and $(w_1,b) = 1$ and choose $w_0, w_2, \cdots, w_n$ such that $-w_0^2 + w_2^2 + \cdots + w_n^2 = (b - w_1^2)$.
  The latter is possible since the quadratic form $-x_0^2 + x_2^2 + \cdots + x_n^2$ represents any arbitrary integer
  ($-x_0^2 + x_2^2$ represents every odd integer, and we have at least one more variable to represent $1$).
  If $w = (w_0, w_1, \dots, w_n)$, we have $\la w, w \ra = b$ whence $w_1^2 > \la w, w \ra > 0$.
  Thus we can construct $M_w = \Gamma_w \bs \H_f$ as in the beginning Section \ref{sec:examples}.
  Since $(w_1^2, b) = 1$, we have 
  \[
  \Z\left[\frac{4(n-1) w_{1}^2}{\la w, w \ra}\right] =  \Z[{4(n-1)}/{b}] = \Z[{2}/{b}] = \Z\left[\frac{2}{\la w, w\ra}\right].
  \]
  Therefore the theorem follows from Corollary \ref{cor:trace-ring}, with $\Gamma_d = \Gamma_w$.
\end{proof}

\begin{remark}\label{rem:nonuniform}
  Since the quadratic form $f$ restricted to $R_1$ is simply
  $-x_0^2 + x_2^2 + \cdots + x_n^2$, it follows (since this form is
  isotropic) that $N_1$ and thus also $M_w$ are non-compact for any
  (admissible) $w$.
\end{remark}
\begin{remark}
  Theorem \ref{th:Z[1/d]} still holds for $n =3$, but the notion of adjoint trace ring/field in the statement differs from its usual meaning
  if the lattices are considered in $\PSL_2(\C)$ (see Remark \ref{rem:n=3-1}).  
  Thus Theorem \ref{th:Z[1/d]} is stated for $n \geq 4$ to avoid confusion.
\end{remark}



\section{Volume bound} \label{sec:volumes}

A doubly-cut glueing such as the $M_w$ from the previous section depends on the choice of finite index subgroup $\Lambda \subset \OO_f'(\Z)$.
Different choices obviously lead to commensurable manifolds; define $V_w$ to be the minimal volume of a manifold in the commensurability class of 
$M_w$.
The goal of this section is to explicitly construct such a $\Lambda$ (satisfying the requirements given in the beginning of Section 
\ref{sec:nonarith}) and give an upper-bound on $V_w$.
The arguments of this section are inspired by the proof of \cite[Lemma 10]{MV}.

We use the notation of the previous section. 
The subgroup $\Lambda \subset \OO_f'(\Z)$ must be torsion-free and satisfy the following: for all $\lambda \in \Lambda$, 
\begin{enumerate}
\item $\lambda R_1 \cap R_1 = \emptyset$, \label{it:req1}
\item $\lambda R_2 \cap R_2 = \emptyset$,\label{it:req2}
\item $\lambda R_1 \cap R_2 = \emptyset$.\label{it:req3}
\end{enumerate}

We focus on congruence subgroups of the form $\Lambda_m = \{\lambda \in \OO_f'(\Z) \sep \lambda \equiv
\id \mod m\}$ (for $m$ not necessarily prime). 
We want to find $m \geq 2$ such that $\Lambda_m$ fulfills our requirements.

To obtain a torsion-free subgroup, it is enough to take a congruence subgroup $\Lambda_m$ with $m > 2$ (see \cite[Theorem IX.7]{Newman}).
Furthermore, it follows from \cite[§2.8]{GPS} that any congruence subgroup satisfies \ref{it:req1}.

For \ref{it:req2}, we need to ensure that for any $\lambda \in \Lambda_m$ we have
\begin{equation} \label{eq:ineq-1}
  |\la \lambda w, w \ra | \geq \la w, w \ra.
\end{equation}
Now if we choose $m \geq 2 \la w,w \ra$, the fact that 
\[
  \la \lambda w, w \ra = \la w, w \ra \mod m
\]
implies inequality \eqref{eq:ineq-1}.

Finally for \ref{it:req3} the situation is similar: we just need to ensure that
\begin{equation} \label{eq:ineq-2}
  |\la \lambda v, w \ra | \geq ||w||.
\end{equation}
If $m \geq 2 w_1^2$, the fact that $w_1^2 \geq \la w, w \ra$ and the equation
\[
  \la \lambda v, w \ra ^2 = \la v, w \ra ^2 = w_1^2 \mod m
\]
imply that the inequalities \eqref{eq:ineq-1} and \eqref{eq:ineq-2} are fulfilled.

To sum up, we define $\Lambda = \Lambda_{m}$ where $m = \max(3, 2w_1^2)$.
This congruence subgroup is then such that $L = \Lambda \bs \H_f$ is a hyperbolic manifold in
which $R_1$ and $R_2$ project to disjoint hypersurfaces, as desired.
Now a doubly-cut glueing $M$ constructed from $\Lambda$ is obtained as the double of a piece of $L$.
Thus the volume of $M$ cannot exceed twice the volume of $L$, and we have proven:
\begin{proposition} \label{prop:vol-M_w}
  The minimal volume $V_w$ of a manifold in the commensurability class of $M_w$ satisfies
  \[
  V_w \leq 2\cdot |\OO_f(\Z/m\Z)| \cdot \covol \OO_f(\Z)
  \]
  where $m=\max(3, 2w_1^2)$.
\end{proposition}

An example of relatively small volume is the following.
\begin{corollary} \label{prop:volumes}
  For any $n > 4$, $n \not \equiv 1 \pmod 3$ there exists an $n$-dimensional nonarithmetic doubly-cut glueing $M$ with
  \[
  \vol(M) \leq 2\cdot |\OO_f(\Z/8\Z)| \cdot \covol \OO_f(\Z).
  \]
\end{corollary}
\begin{proof}
  We pick $w = (1,2,0 ,\dots,0)$. 
  Observe that $w_1^2 = 4 > \la w,w \ra = 3 > 0$.
  Therefore, we can construct a doubly-cut glueing $M_w$ as before.
  Let $M$ be a manifold of minimal volume in the commensurability class of $M_w$.
  The fact that $M$ is nonarithmetic follows from Corollary \ref{cor:trace-ring} (since $\Z[4(n-1)\cdot 4 /3] = \Z[1/3]$
  if $n \not \equiv 1 \pmod 3$) and the volume bound is a consequence of Proposition \ref{prop:vol-M_w}.
\end{proof}

For exact computations of the volumes we refer to the formulas of \cite{RT} and the references therein.

\section{About generalizations} \label{sec:generalizations}
There are two easy to state generalizations of the doubly-cut glueing construction of nonarithmetic manifolds.
The first one is to increase the number of cut hyperplanes.
Indeed, if one finds $n$ disjoint rational hyperplanes, then it is possible (using again \cite[Lemma 3.1]{BT}) to find an arithmetic lattice 
$\Gamma \subset \PO_f$ such that they project down to disjoint hypersurfaces.
If one chooses the hyperplanes in such a way that their reflections have interesting rational properties, one might get 
better results regarding for example the minimal volume estimate of a manifold with prescribed adjoint trace ring.

The second possible generalization is to give an analog of Theorem \ref{th:Z[1/d]} for number fields.
One way to proceed would be to replace $\Z[1/d]$ with the ring $\calO_S$ of $S$-integers of a
totally real number field $k$, where $S$ is a finite set of non-archimedean places.
For specific examples of $S$ this is feasible (as suggested in the toy example of the next proposition).
However in order to get a general result one has to face the problem that an admissible quadratic
form over $k \neq \Q$ is not isotropic, and thus cannot represent any element of $k$.
Moreover, even when we restrict to quadratic extensions, it is likely that more hyperplanes will be needed to generate $\calO_S$, thereby making the construction
more complicated.

\begin{proposition}
  Let $k = \Q(\sqrt{d})$ for $d$ a square-free positive integer.
  Let $n \geq 5$ and $f$ be the quadratic form $f = -\sqrt{d}x_0^2 + x_1^2 + \cdots + x_n^2$.
  Then there exist nonarithmetic lattices in $\OO_f'(\R)$ with trace ring $A$ satisfying 
  \[
  \calO_k \varsubsetneq  A \subset \calO_k\left[\frac{2}{\sqrt{d}- (a^2 + b)}\right],
  \]
  for any integers $a$ and $b$ such that $a^2 + b \geq \sqrt{d}$ and $0 \leq b < \sqrt{d}$.
\end{proposition}
\begin{proof}
  We generalize the arguments of Section \ref{sec:examples} to this case.  

  Let $w = (1, a, b_1, b_2, b_3, b_4, 0, \dots, 0)$, where $b_1^2 + b_2^2 + b_3^2 + b_4^2 = b$
  (Four Squares Theorem).
  The conditions on $a$ and $b$ ensure that $w$ fulfills all the requirements of the beginning of
  Section \ref{sec:examples} (taking norm and scalar product with respect to our new $f$).
  This implies that we can define $M_w = \Gamma_w \bs \H_f$ analogously.
  Moreover, Lemma \ref{lem:Ad-formula} still holds for this choice of $w$ and $f$.

  Corollary \ref{cor:trace-ring} applied to $M_w$ also remains correct if $\Z$ is 
  replaced by $\calO_k$ (since the considered arithmetic group is $\OO_f'(\calO_k)$ instead of
  $\OO_f'(\Z)$).
  As $\la w, w \ra = a ^2 + b - \sqrt{d}$, one just has to remark that
  $4(n-1)\frac{w_1^2}{\la w, w \ra} \notin \calO_k$ and thus $\Gamma_w$ is the desired lattice.
\end{proof}

\bibliographystyle{alpha}
\bibliography{bibliography}

\end{document}